\documentclass[11pt, notitlepage]{article}   

\usepackage{amsmath,amsthm,amsfonts}   
\usepackage{amscd}
\usepackage{amsfonts}
\usepackage{amssymb}
\usepackage{color}      
\usepackage{epsfig}
\usepackage{graphicx}           
\usepackage{hyperref}
\usepackage{tikz}

\hypersetup{colorlinks=true, linkcolor=blue, citecolor=blue, urlcolor=blue}

\theoremstyle{plain}
\newtheorem{theorem}{Theorem}[section]
\newtheorem{lemma}[theorem]{Lemma}
\newtheorem{corollary}[theorem]{Corollary}
\newtheorem{proposition}[theorem]{Proposition}
\newtheorem{observation}[theorem]{Observation}
\newtheorem{remark}[theorem]{Remark}

\begin{document}

\newcommand{\diam}{\textnormal{diam}}

\title{The fractional $k$-metric dimension of graphs}

\author{{\bf{Cong X. Kang}}$^1$, {\bf{Ismael G. Yero}}$^2$ and {\bf{Eunjeong Yi}}$^3$\\
\small Texas A\&M University at Galveston, Galveston, TX 77553, USA$^{1,3}$\\
\small Universidad de C\'{a}diz, Av. Ram\'{o}n Puyol s/n, 11202 Algeciras, Spain$^2$\\
{\small\em kangc@tamug.edu}$^1$; {\small\em ismael.gonzalez@uca.es}$^2$; {\small\em yie@tamug.edu}$^3$}

\maketitle

\date{}

\begin{abstract}
Let $G$ be a graph with vertex set $V(G)$. For any two distinct vertices $x$ and $y$ of $G$, let $R\{x, y\}$ denote the set of vertices $z$ such that the distance from $x$ to $z$ is not equal to the distance from $y$ to $z$ in $G$. For a function $g$ defined on $V(G)$ and for $U \subseteq V(G)$, let $g(U)=\sum_{s \in U}g(s)$. Let $\kappa(G)=\min\{|R\{x,y\}|:  x\neq y \mbox{ and } x,y \in V(G)\}$. For any real number $k \in [1, \kappa(G)]$, a real-valued function $g: V(G) \rightarrow [0,1]$ is a \emph{$k$-resolving function} of $G$ if $g(R\{x,y\}) \ge k$ for any two distinct vertices $x,y \in V(G)$. The \emph{fractional $k$-metric dimension}, $\dim^k_f(G)$, of $G$ is $\min\{g(V(G)): g \mbox{ is a $k$-resolving function of } G\}$. In this paper, we initiate the study of the fractional $k$-metric dimension of graphs. For a connected graph $G$ and $k \in [1, \kappa(G)]$, it's easy to see that $k \le \dim_f^k(G) \le \frac{k|V(G)|}{\kappa(G)}$; we characterize graphs $G$ satisfying $\dim_f^k(G)=k$ and $\dim_f^k(G)=|V(G)|$, respectively. We show that $\dim_f^k(G) \ge k \dim_f(G)$ for any $k \in [1, \kappa(G)]$, and we give an example showing that $\dim_f^k(G)-k\dim_f(G)$ can be arbitrarily large for some $k \in (1, \kappa(G)]$; we also describe a condition for which $\dim_f^k(G)=k\dim_f(G)$ holds. We determine the fractional $k$-metric dimension for some classes of graphs, and conclude with two open problems, including whether $\phi(k)=\dim_f^k(G)$ is a continuous function of $k$ on every connected graph $G$.
\end{abstract}



\section{Introduction}

Let $G$ be a finite, simple, undirected, and connected graph with vertex set $V(G)$ and edge set $E(G)$. For $v \in V(G)$, the \emph{open neighborhood} of $v$ is $N(v)=\{u \in V(G): uv \in E(G)\}$, and the \emph{closed neighborhood} of $v$ is $N[v]=N(v) \cup \{v\}$. The \emph{degree} of a vertex $v \in V(G)$, denoted by $\deg(v)$, is $|N(v)|$; a \emph{leaf} is a vertex of degree one, and a \emph{major vertex} is a vertex of degree at least three. The \emph{distance} between two vertices $x, y \in V(G)$, denoted by $d(x,y)$, is the length of a shortest path between $x$ and $y$ in $G$. The \emph{diameter}, $\diam(G)$, of a graph $G$ is $\max\{d(x,y): x, y \in V(G)\}$. The \emph{complement} of $G$, denoted by $\overline{G}$, is the graph whose vertex set is $V(G)$ and $xy \in E(\overline{G})$ if and only if $xy \not\in E(G)$ for $x,y \in V(G)$. We denote by $K_n$ and $P_n$ the complete graph and the path on $n$ vertices, respectively.\\

For two distinct vertices $x, y \in V(G)$, let $R\{x, y\}=\{z \in V(G): d(x,z) \neq d(y,z)\}$. A subset $S \subseteq V(G)$ is called a \emph{resolving set} of $G$ if $|S \cap R\{x,y\}| \ge 1$ for any two distinct vertices $x$ and $y$ in $G$. The \emph{metric dimension}, $\dim(G)$, of $G$ is the minimum cardinality of $S$ over all resolving sets of $G$. Since metric dimension is suggestive of the dimension of a vector space in linear algebra, sometimes a minimum resolving set of $G$ is called a basis of $G$. The concept of metric dimension was introduced independently by Slater~\cite{slater}, and by Harary and Melter~\cite{HM}. Applications of metric dimension can be found in network discovery and verification~\cite{dim_app3}, robot navigation~\cite{dim_app1}, sonar~\cite{slater}, combinatorial optimization~\cite{dim_app2}, chemistry~\cite{dim_chem}, and strategies for the mastermind game~\cite{dim_app4}. It was noted in~\cite{NPhard} that determining the metric dimension of a graph is an NP-hard problem. Metric dimension has been extensively studied. For a survey on metric dimension in graphs, see~\cite{dim_survey1, dim_survey2}. The effect of the deletion of a vertex or of an edge on the metric dimension of a graph was raised as a fundamental question in graph theory in~\cite{dim_survey2}; the question is essentially settled in~\cite{ve_deletion}.\\

If a minimum number of requisite robots are installed in a network to identify the exact location of an intruder in the network, one malfunctioning robot can lead to failure of detection. Thus, it is natural to build a certain level of redundancy into the detection system. As a generalization of metric dimension, $k$-metric dimension was introduced first by Estrada-Moreno et al.~\cite{kdim1} and, independently, by Adar and Epstein~\cite{kdim3} soon afterwards. Let $\kappa(G)=\min\{|R\{x,y\}|:  x\neq y \mbox{ and } x,y \in V(G)\}$. For a positive integer $k \in \{1,2,\ldots, \kappa(G)\}$, a set $S \subseteq V(G)$ is called a \emph{$k$-resolving set} of $G$ if $|S \cap R\{x,y\}| \ge k$ for any two distinct vertices $x$ and $y$ in $G$. The \emph{$k$-metric dimension}, $\dim^k(G)$\footnote{In fact, the notation of this parameter has been $\dim_k(G)$ in previous works. However, we rather prefer to use $\dim^k(G)$ here, to facilitate the notation of $\dim_f^k(G)$.}, of $G$ is the minimum cardinality over all $k$-resolving sets of $G$. It was shown in~\cite{kdim1} that $k$-metric dimension of a connected graph $G$ exists for every $k \in \{1,2,\ldots, \kappa(G)\}$, and $G$ is called $\kappa(G)$-metric dimensional. For an application of $k$-metric dimension to error-correcting codes, see~\cite{kdim2}. For other articles on the $k$-metric dimension of graphs, see~\cite{kdim4, kdim5, kdim6}.\\

The fractionalization of various graph parameters has been extensively studied (see~\cite{fractionalization}). Currie and Oellermann~\cite{fracdim_o1} defined fractional metric dimension as the optimal solution to a linear programming problem, by relaxing a condition of the integer programming problem for metric dimension. A formulation of fractional metric dimension as a linear programming problem can be found in~\cite{fracdim_o2}. Arumugam and Mathew~\cite{fracdim1} officially studied the fractional metric dimension of graphs. For a function $g$ defined on $V(G)$ and for $U \subseteq V(G)$, let $g(U)=\sum_{s \in U}g(s)$. A real-valued function $g: V(G) \rightarrow [0,1]$ is a \emph{resolving function} of $G$ if $g(R\{x, y\}) \ge 1$ for any two distinct vertices $x, y \in V(G)$. The \emph{fractional metric dimension} of $G$, denoted by $\dim_f(G)$, is $\min\{g(V(G)): g \mbox{ is a resolving function of } G\}$. Notice that $\dim_f(G)$ reduces to $\dim(G)$, if the codomain of resolving functions is restricted to $\{0,1\}$. For more articles on the fractional metric dimension, as well as the closely related fractional strong metric dimension, of graphs, see~\cite{fracdim2, fracdim_w2, fracdim_w1, fracdim_w3, frac_kang, frac_kyy, fracsdim, fracdim_yi}.\\

Next, we introduce fractional $k$-metric dimension, which can be viewed as a generalization of $\dim_f(G)$ as well as a fractionalization of $\dim^k(G)$. For any real number $k \in [1, \kappa(G)]$, a real-valued function $h: V(G) \rightarrow [0,1]$ is a \emph{$k$-resolving function} of $G$ if $h(R\{x, y\}) \ge k$ for any two distinct vertices $x, y \in V(G)$. The \emph{fractional $k$-metric dimension} of $G$, denoted by $\dim_f^k(G)$, is $\min\{h(V(G)): h \mbox{ is a $k$-resolving function of } G\}$; notice that $\dim^1_f(G)=\dim_f(G)$. Note that $\dim_f^k(G)$ reduces to $\dim^k(G)$ when the codomain of $k$-resolving functions is restricted to $\{0,1\}$ and $k \in [1, \kappa(G)]$ is restricted to positive integers.\\

In this paper, we initiate the study of the fractional $k$-metric dimension of graphs. For a connected graph $G$, let $\kappa(G)=\min\{|R\{x,y\}|:  x\neq y \mbox{ and } x,y \in V(G)\}$. The paper is organized as follows. In section~2, we compare $\dim^k_f(G)$ with $\dim^k (G)$ for certain $k$, and we recall some results on the fractional metric dimension of graphs. In section~3, we prove that $\dim_f^k(G) \ge k \dim_f(G)$ for any $k \in [1, \kappa(G)]$; we describe a condition for which $\dim_f^k(G)=k\dim_f(G)$ holds for all $k \in [1, \kappa(G)]$. For $k\in[1, \kappa(G)]$, we show that $k \le \dim_f^k(G) \le \frac{k}{\kappa(G)}|V(G)|$, which implies $k \le \dim_f^k(G) \le |V(G)|$; we characterize graphs $G$ satisfying $\dim_f^k(G)=k$ and $\dim_f^k(G)=|V(G)|$, respectively. In section~4, for $k \in [1, \kappa(G)]$, we determine the fractional $k$-metric dimension of trees, cycles, wheel graphs, the Petersen graph, a bouquet of cycles (i.e., the vertex sum of cycles at one common vertex), complete multi-partite graphs, and grid graphs (i.e., the Cartesian product of two paths). Along the way, we give an example showing that $\dim_f^k(G)-k\dim_f(G)$ can be arbitrarily large for some $k \in (1, \kappa(G)]$. We conclude with some open problems.


\section{Preliminaries}

In this section, we make some observations involving $\dim_f^k(G)$ for $k \in [1, \kappa(G)]$, or $\dim^k(G)$ for $k \in \{1,2,\ldots, \kappa(G)\}$. We also recall some results on the fractional metric dimension of graphs. We begin with some observations. Two distinct vertices $x,y \in V(G)$ are called \emph{twin vertices} if $N(x)-\{y\}=N(y)-\{x\}$.

\begin{observation}
Let $G$ be a connected graph and let $k \in [1, \kappa(G)]$. If two distinct vertices $x$ and $y$ are twin vertices in $G$, then $R\{x, y\}=\{x, y\}$, and thus $\kappa(G)=2$ and $g(x) +g(y) \ge k$ for any $k$-resolving function $g$ of $G$.
\end{observation}

\begin{observation}\label{obs2} Let $G$ be a connected graph.
\begin{itemize}\setlength\itemsep{0em}
\item[(a)] For $k \in [1, \kappa(G)]$, $\dim_f(G) \le \dim_f^k(G)$.
\item[(b)] For $k \in \{1,2,\ldots, \kappa(G)\}$, $\dim_f^k(G) \le \dim^k(G)$.
\item[(c)] \emph{\cite{fracdim1, kdim1}} For $k \in \{1,2,\ldots, \kappa(G)\}$, $\dim_f(G) \le \dim(G) \le \dim^k(G)$.
\end{itemize}
\end{observation}

Observation~\ref{obs2} provides inequalities between any two graph parameters among $\dim_f(G)$, $\dim(G)$, $\dim_f^k(G)$ for $k \in [1, \kappa(G)]$, and $\dim^k(G)$ for $k \in \{1,2,\ldots, \kappa(G)\}$, excluding the relation between $\dim(G)$ and $\dim_f^k(G)$. So, it is natural to compare $\dim(G)$ and $\dim_f^k(G)$ for $k \in [1, \kappa(G)]$.

\begin{remark}
(a) The value of $\dim_f^k(G)-\dim(G)$ can be arbitrarily large, as $G$ varies, for some $k \in (1, \kappa(G)]$. For $n \ge 4$, note that $\dim(P_n)=1 \le \dim_f^k(P_n)$ for $k \in [1, n-1]$ and $\dim_f^{n-1}(P_n)=n$ (see Proposition~\ref{fkdim_path}). Thus, $\dim_f^{n-1}(P_n)-\dim(P_n)=n-1$ can be arbitrarily large.

(b) The value of $\dim(G)-\dim_f^k(G)$ can be arbitrarily large, as $G$ varies, for some $k \in [1, \kappa(G)]$. For $n \ge 3$, note that $\dim(K_n)=n-1$ and $\dim_f^k(K_n)=\frac{kn}{2}$ for $k \in [1,2]$ (see Proposition~\ref{fkdim_multipartite}). Now, let $k\in[1,2)$; then $\dim(K_n)-\dim_f^k(K_n)=\frac{(2-k)n}{2}-1$ becomes arbitrarily large as $n \rightarrow\infty$.
\end{remark}

In light of Observation~\ref{obs2}(b), we have the following

\begin{theorem}
The value of $\dim^k(G)-\dim_f^k(G)$ can be arbitrarily large, as $G$ varies, for some $k\in \{1,2,\ldots, \kappa(G)\}$.
\end{theorem}

\begin{proof}
Let $H$ be a connected graph with vertex set $V(H)=\{u_1, u_2, \ldots, u_n\}$, where $n \ge 2$. Let $G$ be the graph obtained from $H$ as follows:
\begin{itemize}\setlength\itemsep{0em}
\item[(i)] for each $i \in \{1,2,\ldots, n\}$, add three vertices $a_{i,1}, b_{i,1}, c_{i,1}$ and three edges $u_ia_{i,1}, u_ib_{i,1}, u_ic_{i,1}$;
\item[(ii)] for each $i  \in \{1,2,\ldots, n\}$, subdivide the edge $u_ia_{i,1}$ ($u_ib_{i,1}$ and $u_ic_{i,1}$, respectively) exactly $s-1$ times so that the edge $u_ia_{i,1}$ ($u_ib_{i,1}$ and $u_ic_{i,1}$, respectively) in (i) becomes the $u_i-a_{i,1}$ path given by $u_i, a_{i,s}, a_{i,s-1}, \ldots, a_{i,1}$ (the $u_i-b_{i,1}$ path given by $u_i, b_{i,s}, b_{i,s-1}, \ldots, b_{i,1}$ and the $u_i-c_{i,1}$ path given by $u_i, c_{i,s}, c_{i,s-1}, \ldots, c_{i,1}$, respectively).
\end{itemize}
For each $i \in \{1,2,\ldots, n\}$, let $T_i$ be the subtree of $G$ consisting of the $u_i-a_{i,1}$ path, the $u_i-b_{i,1}$ path, and the $u_i-c_{i,1}$ path; further, let $P^{i,a}$ be the $a_{i,s}-a_{i,1}$ path, $P^{i,b}$ the $b_{i,s}-b_{i,1}$ path, and $P^{i,c}$ the $c_{i,s}-c_{i,1}$ path. Then, for each $i \in \{1,2,\ldots, n\}$, 
\begin{equation}\label{e0}
\left\{
\begin{array}{c}
R\{a_{i,s}, b_{i,s}\}=V(P^{i,a}) \cup V(P^{i,b}),\\ 
R\{a_{i,s}, c_{i,s}\}=V(P^{i,a}) \cup V(P^{i,c}),\\ 
R\{b_{i,s}, c_{i,s}\}=V(P^{i,b}) \cup V(P^{i,c}).
\end{array}\right.
\end{equation}
We determine $\kappa(G)$, $\dim^k(G)$ for $k \!\in\! \{1,2,\ldots, \kappa(G)\}$, and $\dim_f^k(G)$ for $k \in [1, \kappa(G)]$.

\vspace{.2cc}

\textbf{Claim 1: $\kappa(G)=2s$.} 

\noindent Proof of Claim 1. Let $x$ and $y$ be two distinct vertices of $G$. First, let $x,y \in V(T_i)$ for some $i \in \{1,2,\ldots, n\}$. If $d(u_i, x) \neq d(u_i, y)$, then $R\{x,y\} \supseteq V(T_j)$ with $|R\{x,y\}| \ge |V(T_j)|=3s+1$, where $j \in\{1,2,\ldots,n\}-\{i\}$. If $d(u_i, x)=d(u_i, y)$, say $x \in V(P^{i,a})$ and $y \in V(P^{i,b})$ without loss of generality, then $R\{x,y\}=V(P^{i,a}) \cup V(P^{i,b})$ with $|R\{x,y\}|=|V(P^{i,a})|+|V(P^{i,b})|=2s$. Second, let $x \in V(T_i)$ and $y\in V(T_j)$ for distinct $i, j \in \{1,2,\ldots, n\}$, say $x$ lies on the $u_i-a_{i,1}$ path and $y$ lies on the $u_j-a_{j,1}$ path, without loss of generality. Then $|R\{x,y\}| \ge 2s$, since at most one vertex lying on a $a_{i,1}-a_{j,1}$ geodesic is at equal distance from both $x$ and $y$. So, $\kappa(G)=2s$.~$\Box$

\vspace{.2cc}

\textbf{Claim 2:} For $k \in \{1,2,\ldots, 2s\}$,
$\dim^k(G)=\left\{
\begin{array}{ll}
\frac{3kn}{2} & \mbox{ if $k$ is even},\\[\smallskipamount]
\frac{(3k+1)n}{2} & \mbox{ if $k$ is odd}.
\end{array}\right.$

\noindent Proof of Claim 2. Let $k \in \{1,2,\ldots, 2s\}$. First, we show that $\dim^k(G) \ge \frac{3kn}{2}$ for an even $k$, and $\dim^k(G) \ge \frac{(3k+1)n}{2}$ for an odd $k$. Let $S$ be a minimum $k$-resolving set of $G$. For each $i \in \{1,2,\ldots, n\}$, (\ref{e0}) implies
\begin{equation}\label{e1}
\left\{
\begin{array}{c}
|S \cap R\{a_{i,s}, b_{i,s}\}|=|S \cap V(P^{i,a})| + |S \cap V(P^{i,b})| \ge k,\\
|S \cap R\{a_{i,s},c_{i,s}\}|=|S \cap V(P^{i,a})| +|S \cap V(P^{i,c})| \ge k,\\
|S \cap R\{b_{i,s}, c_{i,s}\}|=|S \cap V(P^{i,b})|+|S \cap V(P^{i,c})| \ge k.
\end{array}\right.
\end{equation}
Suppose $k$ is even. By summing over the three inequalities in (\ref{e1}), we obtain $|S \cap V(P^{i,a})|+|S \cap V(P^{i,b})|+|S \cap V(P^{i,c})| \ge \frac{3k}{2}$, and thus $|S \cap V(T_i)| \ge \frac{3k}{2}$ for each $i \in \{1,2,\ldots, n\}$. So, $|S \cap V(G)| \ge \sum_{i=1}^{n} \frac{3k}{2}=\frac{3kn}{2}$, and hence $\dim^k(G) \ge \frac{3kn}{2}$. Now, suppose $k$ is odd. If $|S \cap V(P^{i,a})|\le\lfloor\frac{k}{2}\rfloor$ and $|S \cap V(P^{i,b})| \le \lfloor\frac{k}{2}\rfloor$ for some $i \in \{1,2,\ldots, n\}$, then $|S \cap R\{a_{i, s}, b_{i, s}\}|=|S \cap V(P^{i,a})|+|S \cap V(P^{i,b})| \le \lfloor\frac{k}{2}\rfloor+\lfloor\frac{k}{2}\rfloor=k-1$, contradicting the assumption that $S$ is a $k$-resolving set of $G$; thus, $|S \cap V(P^{i,t})| \le \lfloor\frac{k}{2}\rfloor$ for at most one $t \in\{a,b,c\}$ for each $i \in \{1,2,\ldots, n\}$. Let $|S \cap V(P^{i,c})|=\alpha \le \min\{|S \cap V(P^{i,a})|, |S \cap V(P^{i,b})|\}$ for $i \in \{1,2,\ldots, n\}$. Then $|S \cap V(P^{i,a})| \ge k-\alpha$ and $|S \cap V(P^{i,b})| \ge k-\alpha$ from (\ref{e1}), and thus $|S \cap V(T_i)| \ge 2k-\alpha \ge 2k-\lfloor\frac{k}{2}\rfloor=\frac{3k+1}{2}$. So, $|S \cap V(G)| \ge \sum_{i=1}^{n} \frac{3k+1}{2}= \frac{(3k+1)n}{2}$, and hence $\dim^k(G) \ge \frac{(3k+1)n}{2}$.

Second, we show that $\dim^k(G) \le \frac{3kn}{2}$ for an even $k$, and $\dim^k(G) \le \frac{(3k+1)n}{2}$ for an odd $k$. If $k$ is even, let $W_0=\cup_{i=1}^n (\{a_{i,1}, a_{i,2}, \ldots, a_{i, \frac{k}{2}}\} \cup \{b_{i,1}, b_{i,2}, \ldots, b_{i, \frac{k}{2}}\} \cup \{c_{i,1}, c_{i,2}, \ldots, c_{i, \frac{k}{2}}\})$. If $k$ is odd, let $W_1=\cup_{i=1}^n (\{a_{i,1}, a_{i,2}, \ldots, a_{i, \lceil\frac{k}{2}\rceil}\} \cup \{b_{i,1}, b_{i,2}, \ldots, b_{i, \lceil\frac{k}{2}\rceil}\} \cup \{c_{i,1}, c_{i,2}, \ldots, c_{i, \lfloor\frac{k}{2}\rfloor}\})$. Note that $|W_0|=\frac{3kn}{2}$, $|W_1|=\frac{(3k+1)n}{2}$, $|W_0 \cap V(P^{i,a})|=|W_0 \cap V(P^{i,b})|=|W_0 \cap V(P^{i,c})|=\frac{k}{2}$ and $|W_1 \cap V(P^{i,a})|=|W_1 \cap V(P^{i,b})|=\lceil\frac{k}{2}\rceil=1+|W_1 \cap V(P^{i,c})|$ for $i\in\{1,2,\ldots,n\}$. It suffices to show that $W_0$ ($W_1$, respectively) is a $k$-resolving set of $G$ when $k$ is even (odd, respectively). Let $x$ and $y$ be distinct vertices of $G$. Suppose $x,y \in V(T_i)$ for some $i \in \{1,2,\ldots, n\}$. If $d(u_i, x) \neq d(u_i,y)$, then $|W_{\ell} \cap R\{x,y\}| \ge |V(T_j)| \ge \frac{3k}{2} \ge k$ for $j\neq i$ and for $\ell \in \{0,1\}$. If $d(u_i, x)=d(u_i,y)$, say $x\in V(P^{i,a})$ and $y \in V(P^{i,c})$ (other cases can be handled similarly), then $|W_{\ell} \cap R\{x,y\}|=|W_{\ell} \cap (V(P^{i,a}) \cup V(P^{i,c}))| \ge k$ for $\ell\in\{0,1\}$. Now, let $x \in V(T_i)$ and $y \in V(T_j)$ for distinct $i,j\in\{1,2\ldots, n\}$; suppose $d(u_i,x) \le d(u_j,y)$, without loss of generality. Since $R\{x,y\} \supseteq V(T_i)$, $|W_{\ell} \cap R\{x,y\}| \ge |W_{\ell} \cap V(T_i)| \ge \frac{3k}{2}$ for $\ell \in \{0,1\}$. So, $W_0$ ($W_1$, respectively) is a $k$-resolving set of $G$ when $k$ is even (odd, respectively).~$\Box$

\vspace{.2cc}

\textbf{Claim 3:} For $k \in [1, 2s]$, $\dim_f^k(G)=\frac{3kn}{2}$. 

\noindent Proof of Claim 3. Let $k \in [1, 2s]$. First, we show that $\dim_f^k(G) \ge \frac{3kn}{2}$. Let $g: V(G) \rightarrow [0,1]$ be any $k$-resolving function of $G$. From (\ref{e0}), for each $i \in \{1,2,\ldots, n\}$, we have $g(R\{a_{i,s}, b_{i,s}\})\!=\!g(V(P^{i,a}))+ g(V(P^{i,b})) \ge k$, $g(R\{a_{i,s}, c_{i,s}\})\!=\!g(V(P^{i,a}))+ g(V(P^{i,c})) \ge k$, and $g(R\{b_{i,s}, c_{i,s}\})=g(V(P^{i,b}))+ g(V(P^{i,c})) \ge k$. By summing over the three inequalities, we obtain $g(V(P^{i,a}))+ g(V(P^{i,b}))+g(V(P^{i,c})) \ge \frac{3k}{2}$, and thus $g(V(T_i)) \ge \frac{3k}{2}$ for each $i\in\{1,2,\ldots,n\}$. So, $g(V(G)) \ge \sum_{i=1}^{n}\frac{3k}{2}=\frac{3kn}{2}$, and hence $\dim_f^k(G) \ge \frac{3kn}{2}$.

Second, we show that $\dim_f^k(G) \le \frac{3kn}{2}$. Let $h:V(G) \rightarrow [0,1]$ be a function defined by
\begin{equation*}
h(v)=\left\{
\begin{array}{ll}
0 & \mbox{ if } \deg(v) \ge 3,\\
\frac{k}{2s} & \mbox{ if } \deg(v) \le 2.
\end{array}\right.
\end{equation*}
Notice that $h(V(P^{i,a}))=h(V(P^{i,b}))=h(V(P^{i,c}))=\frac{k}{2}$ for each $i \in \{1,2,\ldots, n\}$, and $h(V(G))=\frac{3kn}{2}$. It suffices to show that $h$ is a $k$-resolving function of $G$. Let $x$ and $y$ be distinct vertices of $G$. Suppose $x,y \in V(T_i)$ for some $i \in \{1,2,\ldots, n\}$. If $d(u_i, x) \neq d(u_i,y)$, then $h(R\{x,y\}) \ge h(V(T_j))\ge\frac{3k}{2}$ for $j \neq i$. If $d(u_i, x)=d(u_i,y)$, say $x\!\in\! V(P^{i,a})$ and $y \!\in\! V(P^{i,b})$ without loss of generality, then $h(R\{x, y\})\!=\!h(V(P^{i,a}))+h(V (P^{i,b}))=k$. Now, let $x \in V(T_i)$ and $y \in V(T_j)$ for distinct $i,j\in\{1,2\ldots, n\}$; suppose $d(u_i,x) \le d(u_j,y)$, without loss of generality. Then $h(R\{x, y\}) \ge h(V(T_i))\ge \frac{3k}{2}$. So, $h$ is a $k$-resolving function of $G$.~$\Box$

\vspace{.2cc}

By Claims 2 and~3, we see that, for each odd $k\in \{1,2,\ldots, 2s\}$, $\dim^k(G)-\dim_f^k(G)=\frac{(3k+1)n}{2}-\frac{3kn}{2}=\frac{n}{2}$, which can be arbitrarily large.~\hfill
\end{proof}

Next, we recall some results on the fractional metric dimension of graphs. One can easily see that, for any connected graph $G$ of order at least two, $1 \le \dim_f(G) \le \frac{|V(G)|}{2}$ (see~\cite{fracdim1}). For the characterization of graphs $G$ achieving the lower bound, see Theorem~\ref{thm_frac}(a). Regarding the characterization of graphs $G$ achieving the upper bound, the following result is stated in~\cite{fracdim1} and a correct proof is provided in~\cite{frac_kang}.

\begin{theorem}~\emph{\cite{fracdim1, frac_kang}}
Let $G$ be a connected graph of order at least two. Then $\dim_f(G)=\frac{|V(G)|}{2}$ if and only if there exists a bijection $\alpha: V(G) \rightarrow V(G)$ such that $\alpha(v) \neq v$ and $|R\{v, \alpha(v)\}|=2$ for all $v \in V(G)$.
\end{theorem}

An explicit characterization of graphs $G$ satisfying $\dim_f(G)=\frac{|V(G)|}{2}$ is given in~\cite{fracdim2}. We recall the following construction from~\cite{fracdim2}. Let $\mathcal{K}=\{K_n: n \ge 2\}$ and $\overline{\mathcal{K}}=\{\overline{K}_n: n \ge 2\}$. Let $H[\mathcal{K} \cup \overline{\mathcal{K}}]$ be the family of graphs obtained from a connected graph $H$ by (i) replacing each vertex $u_i \in V(H)$ by a graph $H_i \in \mathcal{K} \cup \overline{\mathcal{K}}$, and (ii) each vertex in $H_i$ is adjacent to each vertex in $H_j$ if and only if $u_iu_j \in E(H)$.

\begin{theorem}~\emph{\cite{fracdim2}}\label{n_2}
Let $G$ be a connected graph of order at least two. Then $\dim_f(G)=\frac{|V(G)|}{2}$ if and only if $G \in H[\mathcal{K} \cup \overline{\mathcal{K}}]$ for some connected graph $H$.
\end{theorem}

Now, we recall the fractional metric dimension of some classes of graphs. We begin by recalling some terminologies. Fix a graph $G$. A leaf $u$ is called a \emph{terminal vertex} of a major vertex $v$ if $d(u,v) < d(u,w)$ for every other major vertex $w$. The terminal degree, $ter_G(v)$, of a major vertex $v$ is the number of terminal vertices of $v$. A major vertex $v$ is an \emph{exterior major vertex} if it has positive terminal degree. Let $ex(G)$ denote the number of exterior major vertices of $G$, $ex_a(G)$ the number of exterior major vertices $u$ with $ter_G(u)=a$, and $\sigma(G)$ the number of leaves of $G$.

\begin{theorem}\label{thm_frac}
\setlength\itemsep{0em}
\item[(a)] \emph{\cite{fracsdim}} For any graph $G$ of order $n \ge 2$, $\dim_f(G)=1$ if and only if $G \cong P_n$.
\item[(b)] \emph{\cite{fracdim_yi}} For a tree $T$, $\dim_f(T)=\frac{1}{2}(\sigma(T)-ex_1(T))$.
\item[(c)] \emph{\cite{fracdim1}} For the Petersen graph $\mathcal{P}$, $\dim_f(\mathcal{P})=\frac{5}{3}$.
\item[(d)] \emph{\cite{fracdim1}} For an $n$-cycle $C_n$,
$\dim_f(C_n)=\left\{
\begin{array}{ll}
\frac{n}{n-2} & \mbox{if $n$ is even},\\
\frac{n}{n-1} & \mbox{if $n$ is odd}.
\end{array}\right.$
\item[(e)] \emph{\cite{fracdim1}} For the wheel graph $W_n$ of order $n \ge 5$,
$\dim_f(W_n)=\left\{
\begin{array}{ll}
2 & \mbox{if } n=5,\\
\frac{3}{2} & \mbox{if } n=6,\\
\frac{n-1}{4} & \mbox{if } n \ge 7.
\end{array}\right.$
\item[(f)] \emph{\cite{fracsdim}} If $B_m$ is a bouquet of $m$ cycles with a cut-vertex (i.e., the vertex sum of $m$ cycles at one common vertex), where $m\ge 2$, then $\dim_f(B_m)=m$.
\item[(g)] \emph{\cite{fracdim_yi}} For $m \ge 2$, let $G=K_{a_1,a_2, \ldots, a_m}$ be a complete $m$-partite graph of order $n=\sum_{i=1}^{m}a_i$, and let $s$ be the number of partite sets of $G$ consisting of exactly one element. Then
\begin{equation*}
\dim_f(G)=\left\{
\begin{array}{ll}
\frac{n-1}{2} & \mbox{if }s=1,\\
\frac{n}{2} & \mbox{otherwise}.
\end{array}\right.
\end{equation*}
\item[(h)]  \emph{\cite{fracdim1}} For the grid graph $G=P_s \square P_t$ ($s,t \ge 2$), $\dim_f(G)=2$.
\end{theorem}


\section{Some general results on fractional $k$-metric dimension}

In this section, we show that $\dim_f^k(G) \ge k \dim_f(G)$ for any $k \in [1, \kappa(G)]$. We also describe a condition for which 
$\dim_f^k(G)=k\dim_f(G)$ holds for all $k \in [1, \kappa(G)]$. For all $k \in [1,\kappa(G)]$, we show that $k \le \dim_f^k(G) \le \frac{k}{\kappa(G)}|V(G)|$, which implies $k \le \dim_f^k(G) \le |V(G)|$; we characterize graphs $G$ satisfying $\dim_f^k(G)=k$ and $\dim_f^k(G)=|V(G)|$, respectively. We conclude with an example such that two non-isomorphic graphs $H_1$ and $H_2$ satisfy $\dim^k_f(H_1)=\dim_f^k(H_2)$ for all $k \in [1, \kappa]$, where $\kappa(H_1)=\kappa(H_2)=\kappa$.

We begin by comparing the fractional metric dimension and the fractional $k$-metric dimension of graphs.

\begin{lemma}\label{fdim_fkdim}
For any connected graph $G$ and for any $k \in [1,\kappa(G)]$, $\dim_f^k(G) \ge k \dim_f(G)$.
\end{lemma}

\begin{proof}
Let $g: V(G) \rightarrow [0,1]$ be a minimum $k$-resolving function of $G$. Then $g(R\{x,y\}) \ge k$ for any two distinct vertices $x,y \in V(G)$. Now, let $h: V(G) \rightarrow [0,1]$ be a function defined by $h(u)=\frac{1}{k}g(u)$ for each $u \in V(G)$. Then $h(R\{x,y\})=\frac{1}{k} g(R\{x,y\})\ge 1$ for any distinct vertices $x,y \in V(G)$; thus, $h$ is a resolving function of $G$. So, $h(V(G))=\frac{1}{k} \dim^k_f(G) \ge \dim_f(G)$, i.e., $\dim^k_f(G) \ge k \dim_f(G)$.~\hfill
\end{proof}

Next, we examine the conditions for which $\dim_f^k(G)=k \dim_f(G)$ holds.

\begin{lemma}\label{fdim_fkdim2}
Let $G$ be a connected graph and let $k \in[1, \kappa(G)]$. If there exists a minimum resolving function $g: V(G) \rightarrow [0,1]$ such that $g(v) \le \frac{1}{k}$ for each $v \in V(G)$, then $\dim_f^k(G)=k\dim_f(G)$ for any $k \in [1, \kappa(G)]$.
\end{lemma}

\begin{proof}
Let $g: V(G) \rightarrow [0,1]$ be a minimum resolving function of $G$ satisfying $g(v) \le \frac{1}{k}$ for each $v \in V(G)$. Let $h: V(G) \rightarrow [0,1]$ be a function defined by $h(v)=kg(v)$ for each $v \in V(G)$. Then $h$ is a $k$-resolving function of $G$: (i) for each $v \in V(G)$, $0 \le h(v)=kg(v) \le 1$; (ii) for any two distinct $x,y \in V(G)$, $h(R\{x,y\})=kg(R\{x,y\}) \ge k$, since $g(R\{x,y\}) \ge 1$ by the assumption that $g$ is a resolving function of $G$. So, $\dim_f^k(G) \le h(V(G))=kg(V(G))=k\dim_f(G)$. Since $\dim_f^k(G) \ge k \dim_f(G)$ by Lemma~\ref{fdim_fkdim}, $\dim_f^k(G)=k\dim_f(G)$.~\hfill
\end{proof}

Next, we obtain the lower and upper bounds of $\dim_f^k(G)$ in terms of $k$, $\kappa(G)$, and the order of $G$.

\begin{proposition}\label{bounds1}
Let $G$ be a connected graph of order $n$. For any $k \in [1, \kappa(G)]$, $\displaystyle k \le \dim_f^k(G) \le \frac{kn}{\kappa(G)},$ where both bounds are sharp.
\end{proposition}

\begin{proof}
The lower bound is trivial. For the upper bound, let $g:V(G)\rightarrow [0,1]$ be a function such that $g(v)=\frac{k}{\kappa(G)}$ for each $v\in V(G)$. Since $\kappa(G)=\min\{|R\{x,y\}|:  x\neq y \mbox{ and } x,y \in V(G)\}$, $g(R\{u,v\}) \ge k$ for any distinct vertices $u,v\in V(G)$. So, $g$ is a $k$-resolving function of $G$, and hence $\dim_f^k(G) \le g(V(G))= \sum_{i=1}^n \frac{k}{\kappa(G)}= \frac{kn}{\kappa(G)}$.

For the sharpness of the lower bound, see Proposition~\ref{characterization}(a); for the sharpness of the upper bound, we refer to Proposition~\ref{fkdim_cycle}.~\hfill
\end{proof}

As an immediate consequence of Proposition~\ref{bounds1}, we have the following.

\begin{corollary}\label{cor_bounds1}
For a connected graph $G$ of order $n$ and for $k \in [1, \kappa(G)]$, $\displaystyle k \le \dim_f^k(G) \le n$.
\end{corollary}

Next, we characterize graphs $G$ achieving the lower bound and the upper bound, respectively, of Corollary~\ref{cor_bounds1}. Let $\mathcal{R}_{\kappa}(G)= \displaystyle\bigcup_{x,y \in V(G), x \neq y, |R\{x,y\}|=\kappa} R\{x,y\}$, where $\kappa=\kappa(G)$.

\begin{proposition}\label{characterization}
For a connected graph $G$ of order $n \ge 2$ and for $k \in [1, \kappa(G)]$,
\begin{itemize}\setlength\itemsep{0em}
\item[(a)] $\dim_f^k(G)=k$ if and only if $G \cong P_n$ and $k \in [1,2]$,
\item[(b)] $\dim_f^k(G)=n$ if and only if $k=\kappa(G)=\kappa$ and $V(G)=\mathcal{R}_{\kappa}(G)$.
\end{itemize}
\end{proposition}

\begin{proof}
\textbf{(a)} ($\Leftarrow$) If $G \cong P_n$ and $k \in [1,2]$, then $\dim_f^k(G)=k$ by Proposition~\ref{fkdim_path}.

($\Rightarrow$) Suppose that $\dim_f^k(G)=k$. Since $\dim_f^k(G) \ge k \dim_f(G) \ge k$ by Lemma~\ref{fdim_fkdim}, $\dim_f^k(G)=k$ implies $\dim_f(G)=1$; thus $G \cong P_n$ by Theorem~\ref{thm_frac}(a) and $k \in [1,2]$ from Proposition~\ref{fkdim_path}.

\vspace{.2cc}

\textbf{(b)} ($\Leftarrow$) Let $k=\kappa(G)=\kappa$ and $V(G)=\mathcal{R}_{\kappa}(G)$. Then, for any vertex $v\in V(G)$, there exist two
distinct vertices $x,y\in V(G)$ such that $v\in R\{x,y\}$ with $|R\{x,y\}|=\kappa$. Since any $\kappa$-resolving function $g$ of $G$ must satisfy $g(R\{x,y\}) \ge \kappa$ and $g(v) \le 1$ for each $v \in V(G)$, $g(u)=1$ for each $u \in R\{x,y\}$ with $|R\{x,y\}|=\kappa$. Since $V(G)=\mathcal{R}_{\kappa}(G)$, $g(v)=1$ for each $v \in V(G)$. Thus, $g(V(G))=n$ and hence $\dim_f^k(G)=n$.

  ($\Rightarrow$) Let $\dim_f^k(G)=n$. By Proposition~\ref{bounds1}, $k=\kappa$. Suppose that $\mathcal{R}_{\kappa}(G)\subsetneq V(G)$ and let $w\in V(G)-\mathcal{R}_{\kappa}(G)$. Then, for any vertex $w' \in V(G)-\{w\}$, $|R\{w, w'\}| \ge \kappa(G)+1$. If $h: V(G) \rightarrow [0,1]$ is a function defined by $h(w)=0$ and $h(v)=1$ for each $v \in V(G)-\{w\}$, then $h$ is a $\kappa$-resolving function of $G$ with $h(V(G))=n-1$, which contradicts the assumption that $\dim_f^k(G)=n$.~\hfill
\end{proof}

\begin{remark}
Let $G \in H[\mathcal{K} \cup \overline{\mathcal{K}}]$ for some connected graph $H$, as described in Theorem~\ref{n_2}. Then $\kappa(G)=2$ and $V(G)=\mathcal{R}_{\kappa}(G)$; thus $\dim_f^2(G)=|V(G)|$ by Proposition~\ref{characterization}(b). More generally, $\dim_f^k(G)=k\dim_f(G)=\frac{k}{2}|V(G)|$ for $k \in [1,2]$ by Lemma~\ref{fdim_fkdim2}, since $g:V(G) \rightarrow [0,1]$ defined by $g(u)=\frac{1}{2}$, for each $u \in V(G)$, forms a minimum resolving function of $G$.
\end{remark}

We conclude this section with an example showing that two non-isomorphic graphs can have the same $\kappa$ and identical $k$-fractional metric dimension for all $k \in[1,\kappa]$.

\begin{remark}
There exist non-isomorphic graphs $H_1$ and $H_2$ such that $\dim_f^k(H_1)=\dim_f^k(H_2)$ for all $k \in [1, \kappa]$, where $\kappa=\kappa(H_1)=\kappa(H_2)$. For example, let $H_1 \cong K_{2,2,3}$ and $H_2 \cong K_{3,4}$; then $\dim_f(H_1)=\dim_f(H_2)=\frac{7}{2}$ by Theorem~\ref{thm_frac}(g). Since both $H_1$ and $H_2$ have twin vertices, $\kappa(H_1)=\kappa(H_2)=2$. Also note that a function $g_i: V(H_i) \rightarrow [0,1]$ defined by $g_i(u)=\frac{1}{2}$ for each $u \in V(H_i)$ forms a minimum resolving function for $H_i$, where $i \in \{1,2\}$. By Lemma~\ref{fdim_fkdim2}, $\dim_f^k(H_1)=\dim_f^k(H_2)=\frac{7}{2}k$ for every $k \in [1, \kappa]$, whereas $H_1 \not\cong H_2$.
\end{remark}


\section{The fractional $k$-metric dimension of some graphs}

In this section, we determine $\dim_f^k(G)$ for $k \in [1, \kappa(G)]$ when $G$ is a tree, a cycle, a wheel graph, the Petersen graph, a bouquet of cycles, a complete multi-partite graph, or a grid graph (the Cartesian product of two paths). Along the way, we provide an example showing that $\dim_f^k(G)-k\dim_f(G)$ can be arbitrarily large for some $k \in (1, \kappa(G)]$. First, we determine $\dim_f^k(G)$ when $G$ is a path.

\begin{proposition}\label{fkdim_path}
Let $P_n$ be an $n$-path, where $n \ge 2$. Then $\dim_f^k(P_2)=k$ for $k \in [1,2]$ and, for $n \ge 3$,
\begin{equation*}
\dim_f^k(P_n)=\left\{
\begin{array}{ll}
k & \mbox{ if } k \in [1,2],\\
2+(k-2)\frac{n-2}{n-3} & \mbox{ if } k \in (2,n-1].
\end{array}\right.
\end{equation*}
\end{proposition}

\begin{proof}
Let $P_n$ be an $n$-path given by $u_1, u_2, \ldots, u_n$, where $n \ge 2$; then $\kappa(P_2)=2$ and $\kappa(P_n)=n-1$ for $n \ge 3$. Since a function $h$ defined on $V(P_2)$ by $h(u_1)=h(u_2)=\frac{1}{2}$ is a minimum resolving function of $P_2$, $\dim_f^k(P_2)=k\dim_f(P_2)=k$ for $k \in [1,2]$ by Theorem~\ref{thm_frac}(a) and Lemma~\ref{fdim_fkdim2}. So, let $n \ge 3$ and we consider two cases.

\textbf{Case 1: $k \in [1,2]$.} If $g: V(P_n) \rightarrow [0,1]$ is a function defined by $g(u_1)=g(u_n)=\frac{1}{2}$ and $g(u_i)=0$ for each $i \in \{2,\ldots, n-1\}$, then $g$ is a minimum resolving function of $P_n$: (i) for any two distinct vertices $x,y \in V(P_n)$, $R\{x,y\} \supseteq \{u_1, u_n\}$, and hence $g(R\{x,y\}) \ge g(u_1)+g(u_n)=\frac{1}{2}+\frac{1}{2}=1$; (ii) $g(V(P_n))=1=\dim_f(P_n)$ by Theorem~\ref{thm_frac}(a). Since $g(u_i) \le \frac{1}{2} \le \frac{1}{k}$ for each $i \in \{1,2,\ldots, n\}$, we have $\dim_f^k(P_n)=k \dim_f(P_n)=k$ for any $k \in [1,2]$ by Lemma~\ref{fdim_fkdim2} and Theorem~\ref{thm_frac}(a).

\textbf{Case 2: $k \in (2,n-1]$.} Note that $n \ge 4$ in this case since $\kappa(P_3)=2$. Let $h: V(P_n) \rightarrow [0,1]$ be a $k$-resolving function of $P_n$. Let $h(u_1)+h(u_n)=b$; then $0 \le b \le 2$. Since $R\{u_i, u_{i+2}\}=V(P_n)-\{u_{i+1}\}$ for $i \in \{1,2,\ldots, n-2\}$, $h(R\{u_i, u_{i+2}\})=h(V(P_n))-h(u_{i+1}) \ge k$ for each $i \in \{1,2,\ldots, n-2\}$. By summing over the $(n-2)$ inequalities, we have $(n-2)h(V(P_n))-\sum_{j=2}^{n-1}h(u_j) \ge (n-2)k$, i.e., $(n-3)h(V(P_n))+h(u_1)+h(u_n) \ge (n-2)k$. So, $h(V(P_n)) \ge \frac{(n-2)k-b}{n-3}$ since $n>3$; note that the minimum of $h(V(P_n))$ is $\frac{(n-2)k-2}{n-3}$ when $b=h(u_1)+h(u_n)$ takes the maximum value 2. Thus, $\dim_f^k(P_n) \ge \frac{(n-2)k-2}{n-3}=2+(k-2)\frac{n-2}{n-3}$.

Now, let $g: V(P_n) \rightarrow [0,1]$ be a function defined by
\begin{equation*}
g(u_i)=\left\{
\begin{array}{ll}
1 & \mbox{ if } i \in \{1,n\},\\
\frac{k-2}{n-3} & \mbox{ otherwise}.
\end{array}\right.
\end{equation*}
Then $g$ is a $k$-resolving function of $P_n$: (i) $0<\frac{k-2}{n-3} \le 1$ for any $k \in (2, n-1]$; (ii) for any distinct $i,j \in \{1,2,\ldots,n\}$, $g(R\{u_i, u_j\}) \ge 2+(n-3)\frac{k-2}{n-3}=k$ since $R\{u_i, u_j\} \supseteq \{u_1, u_n\}$ and $|R\{u_i, u_j\}|\ge n-1$. So, $\dim_f^k(P_n) \le g(V(P_n))=2+(n-2)\frac{k-2}{n-3}$.

Therefore, $\dim_f^k(P_n)=2+(k-2)\frac{n-2}{n-3}$ for $k \in (2,n-1]$.~\hfill
\end{proof}

Second, we determine $\dim_f^k(T)$ when $T$ is a tree that is not a path. Let $M(T)$ be the set of exterior major vertices of a tree $T$. Let $M_1(T)=\{w \in M(T): ter_T(w)=1\}$, $M_2(T)=\{w \in M(T): ter_T(w)=2\}$, and $M_3(T)=\{w \in M(T): ter_T(w) \ge 3\}$; note that $M(T)=M_1(T) \cup M_2(T) \cup M_3(T)$. For any vertex $v\in M(T)$, let $T_v$ be the subtree of $T$ induced by $v$ and all vertices belonging to the paths joining $v$ with its terminal vertices. Let $M^*(T)=M_2(T) \cup M_3(T)$. We recall the following result on $\kappa(T)$.

\begin{theorem}\emph{\cite{kdim1}}\label{kappa_tree} 
For a tree $T$ that is not a path, 
$$\kappa(T)=\min \displaystyle\bigcup_{w \in M^*(T)} \{d(\ell_i, \ell_j): \ell_i \mbox{ and } \ell_j \mbox{ are two distinct terminal vertices of }w\}.$$
\end{theorem}

We begin by examining $\dim_f^k(T)$ when $T$ is a tree with exactly one exterior major vertex. 

\begin{proposition}\label{fkdim_tree} 
Let $T$ be a tree with $ex(T)=1$. Let $v$ be the exterior major vertex of $T$ and let $\ell_1, \ell_2, \ldots, \ell_a$ be the terminal vertices of $v$ in $T$ (note that $a \ge 3$). Suppose that $d(v, \ell_1) \le d(v, \ell_2) \le \ldots \le d(v, \ell_a)$. Then $\kappa(T)=d(\ell_1, \ell_2)$, and
\begin{itemize}\setlength\itemsep{0em}
\item[(a)] if $d(v, \ell_1)=d(v, \ell_2)$, then $\dim_f^k(T)=k\dim_f(T)=\frac{ka}{2}$ for $k \in [1, \kappa(T)]$;
\item[(b)] if $d(v, \ell_1) < d(v, \ell_2)$, then
\begin{equation*}
\dim_f^k(T)=\left\{
\begin{array}{ll}
\frac{ka}{2} & \mbox{ for } k\in [1, 2d(v, \ell_1)],\\
(a-1)k-(a-2)d(v, \ell_1) & \mbox{ for } k \in (2d(v, \ell_1), \kappa(T)].
\end{array}\right.
\end{equation*}
\end{itemize}
\end{proposition}

\begin{proof}
For each $i \in \{1,2,\ldots, a\}$, let $s_i$ be the neighbor of $v$ lying on the $v-\ell_i$ path, and let $P^i$ denote the $s_i-\ell_i$ path in $T$. By Theorem~\ref{kappa_tree}, $\kappa(T)=d(\ell_1, \ell_2)$. Let $k \in [1, \kappa(T)]$.

\textbf{(a)} Let $d(v, \ell_1)\!=\!d(v, \ell_2)$. By Lemma~\ref{fdim_fkdim}, $\dim_f^k(T)\!\ge\! k\dim_f(T)$. We will show that $\dim_f^k(T)\!\le\! k\dim_f(T)$. Let $g: V(T) \rightarrow [0,1]$ be a function defined by
\begin{equation}\label{tree_ex1}
g(u)=\left\{
\begin{array}{ll}
0 & \mbox{ if } u=v,\\
\frac{k}{2d(v, \ell_i)} & \mbox{ for each vertex } u \in V(P^i), \mbox{ where } i \in \{1,2,\ldots, a\}.
\end{array}\right.
\end{equation}
Note that, for each $i \in \{1,2,\ldots, a\}$, (i) $g(V(P^i))=\frac{k}{2}$; (ii) $0 \le \frac{k}{2d(v, \ell_i)} \le 1$ since $k \le \kappa(T)=
d(\ell_1, \ell_2)=2d(v, \ell_1) \le 2d(v, \ell_i)$. If two distinct vertices $x$ and $y$ lie on the $v-\ell_i$ path for some $i \in \{1,2,\ldots, a\}$, then $g(R\{x,y\}) \ge g(V(T)-V(P^i)) \ge k$ since $a \ge 3$. For distinct $i,j \in \{1,2,\ldots, a\}$, if $x \in V(P^i)$ and $y \in V(P^j)$ with $d(v,x) \neq d(v,y)$, say $d(v,x)<d(v,y)$ without loss of generality, then $g(R\{x,y\}) \ge g(V(T)-V(P^j)) \ge k$; if $x \in V(P^i)$ and $y \in V(P^j)$ with $d(v,x)=d(v,y)$, then $g(R\{x,y\}) =g(V(P^i) \cup V(P^j))=k$. So, $g$ is a $k$-resolving function of $T$ with $g(V(T))=\frac{ka}{2}$. Thus $\dim^k_f(T) \le \frac{ka}{2}=k \dim_f(T)$ by Theorem~\ref{thm_frac}(b). Therefore, $\dim_f^k(T) = k\dim_f(T)=\frac{ka}{2}$ for $k \in [1, \kappa(T)]$.

\vspace{.2cc}

\textbf{(b)} Let $d(v, \ell_1)<d(v, \ell_2)$, and we consider two cases.

\textbf{Case 1: $k \in [1, 2d(v, \ell_1)]$.} In this case, the function $g$ in (\ref{tree_ex1}) is a $k$-resolving function of $T$ as shown in the proof for (a); thus, $\dim_f^k(T)\le g(V(T))=\frac{ka}{2}=k\dim_f(T)$. Since $\dim_f^k(T) \ge k\dim_f(T)$ by Lemma~\ref{fdim_fkdim}, $\dim_f^k(T)=k\dim_f(T)=\frac{ka}{2}$.

\textbf{Case 2: $k \in (2d(v, \ell_1), \kappa(T)]$.} Let $h: V(T) \rightarrow [0,1]$ be a minimum $k$-resolving function of $T$. Note that (i) $h(V(P^1)) \le d(v, \ell_1)$ since $h(u) \le 1$ for each $u\in V(T)$; (ii) for distinct $i,j \in \{1,2,\ldots,a\}$, $h(V(P^i))+h(V(P^j)) \ge k$ since $R\{s_i,s_j\}=V(P^i)\cup V(P^j)$. Let $h(V(P^1))=\beta$. From $h(V(P^1))+h(V(P^j)) \ge k$ for each $j \in \{2,3,\ldots, a\}$, $h(V(P^j))\ge k-\beta$. So, $h(V(T)) \ge h(V(P^1)) + \sum_{i=2}^{a} h (V(P^i)) \ge \beta+(a-1)(k-\beta)
=(a-1)k-(a-2)\beta \ge (a-1)k-(a-2)d(v, \ell_1)$ since $\beta \le d(v, \ell_1)$; thus $\dim_f^k(T) \ge (a-1)k-(a-2)d(v, \ell_1)$.

Next, we show that $\dim_f^k(T) \le (a-1)k-(a-2)d(v, \ell_1)$. Let $g: V(T) \rightarrow [0,1]$ be a function defined by 
\begin{equation*}
g(u)=\left\{
\begin{array}{ll}
0 & \mbox{ if } u=v,\\
1 & \mbox{ for each vertex } u \in V(P^1),\\
\frac{k-d(v, \ell_1)}{d(v, \ell_j)} & \mbox{ for each vertex } u \in V(P^j), \mbox{ where } j \in \{2,\ldots, a\}.
\end{array}\right.
\end{equation*}
Note that (i) $g(V(P^1))=d(v, \ell_1)$; (ii) for $j \in \{2,3,\ldots, a\}$, $g(V(P^j))=k-d(v, \ell_1)>k-\frac{k}{2}=\frac{k}{2}$ since $k > 2d(v, \ell_1)$; (iii) $g(V(T))=d(v, \ell_1)+(a-1)(k-d(v, \ell_1))=(a-1)k-(a-2)d(v, \ell_1)$. Also note that $g$ is a $k$-resolving function of $T$: (i) $0 \le \frac{d(v, \ell_1)}{d(v, \ell_j)} \le \frac{k-d(v, \ell_1)}{d(v, \ell_j)} \le \frac{d(\ell_1, \ell_2)-d(v, \ell_1)}{d(v, \ell_j)}=\frac{d(v, \ell_2)}{d(v, \ell_j)}\le 1$ for $j \in \{2,\ldots, a\}$; (ii) if two distinct vertices $x$ and $y$ lie on the $v-\ell_i$ path for some $i \in \{1,2,\ldots, a\}$, then $g(R\{x,y\}) \ge g(V(T))-g(V(P^i)) \ge \min\{k, 2(k-d(v, \ell_1))\} = k$ since $a \ge 3$ and $k-d(v, \ell_1) > \frac{k}{2}$; (iii) if $x \in V(P^i)$ and $y \in V(P^j)$ with $d(v,x) \neq d(v,y)$ for distinct $i,j\in \{1,2,\ldots,a\}$, then $g(R\{x,y\}) \ge \min\{g(V(T)-V(P^i)), g(V(T)-V(P^j))\} \ge k$, since at most one vertex in the $\ell_i-\ell_j$ path can be at equal distance from both $x$ and $y$ in $T$; (iv) if $x \in V(P^i)$ and $y \in V(P^j)$ with $d(v,x) = d(v,y)$ for 
distinct $i,j\in \{1,2,\ldots,a\}$, then $g(R\{x,y\})=g(V(P^i))+g(V(P^j)) \ge \min\{k, 2(k-d(v, \ell_1))\} =k$. Thus, $\dim_f^k(T) \le g(V(T))=(a-1)k-(a-2)d(v, \ell_1)$.

Therefore, $\dim_f^k(T)=(a-1)k-(a-2)d(v, \ell_1)$.~\hfill
\end{proof}

Next, we determine $\dim_f^k(T)$ for a tree $T$ with $ex(T) \ge 1$. We begin with the following lemma, which, besides being useful for Theorem~\ref{tree_general}, bears independent interest.

\begin{lemma}\label{lemma_tree}
Let $T$ be a tree with $ex(T) \ge 2$. For $w\in M^*(T)$, let $x\in V(T_w)$ and $y\in V(T)-V(T_w)$. Then either $R\{x,y\} \supseteq V(T_w)$ or $R\{x,y\} \supseteq V(T_{w'})$ for some $w'\in M^*(T)-\{w\}$. 
\end{lemma}

\begin{proof}
Since $ex(T) \ge 2$, $|M^*(T)| \ge 2$. Let $x\in V(T_w)$ for some $w\in M^*(T)$. First, suppose $y\in V(T_{w'})$ for some $w'\in M^*(T)-\{w\}$. Assume, for contradiction, that there exist $u\in V(T_w)$ and $v\in V(T_{w'})$ such that 
\begin{equation}\label{lem_eq1}
d(x,u)=d(y,u) \mbox{ and } d(x,v)=d(y,v).
\end{equation} 
Put $a=d(y,v), b=d(v,w'), c=d(w',w), d=d(w,u), e=d(u,x)$. Let us call a path leading from $w$ to any of its leaves a ``$w$-terminal path". We may assume that $u$ and $x$ lie in the same $w$-terminal path, since $d(x,u)=d(y,u)$ implies $d(x,w)=d(y,w)$ if $u$ and $x$ lie in distinct $w$-terminal paths. Likewise, we assume $v$ and $y$ lie in the same $w'$-terminal path. After writing the two equations (\ref{lem_eq1}) in terms of components $a,b,c,d,e$ and simplifying, we obtain $b+c+d=0$. This means that $b=c=d=0$, since all variables denote (nonnegative) distances. In particular, the distinctness of $w$ and $w'$ is contradicted by $c=0$. 

Now, suppose $y\notin V(T_{z})$ for any $z\in M^*(T)$; then either $y \in V(T_{y'})$ for some $y' \in M_1(T)$ or $y \not\in V(T_{z'})$ for any $z' \in M(T)$. Note that there exists a vertex $w' \in M^*(T)-\{w\}$ such that either $y$ or $y'$ lies in the $w-w'$ path in $T$. Since $d(s,x)=d(s,y)$ for $s\in V(T_{w'})$ is equivalent to $d(w',x)=d(w',y)$, we may assume, for contradiction, that 
\begin{equation}\label{lem_eq2}
d(x,w')=d(y,w') \mbox{ and } d(x,t)=d(y,t), \mbox{ where }t\in V(T_w).
\end{equation} 
Put $d(y,w')=a+b$ and $d(w,w')=c+b$, where $b\geq 0$ denotes the length of the path shared between the $w-w'$ path and the $y-w'$ path. Similarly, put $d(w,t)=d$ and $d(x,t)=e$. As before, we may assume that $x$ and $t$ lie in the same $w$-terminal path. After simplifying the two equations (\ref{lem_eq2}) in terms of $a,b,c,d,e$, we obtain $c=d=0$. This implies that either $y\in V(T_w)$ (when $b>0$) or $w=w'$ (when $b=0$); both possibilities contradict the present assumptions.~\hfill 
\end{proof}

\begin{theorem}\label{tree_general}
Let $T$ be a tree with $ex(T) \ge 1$. Then $\kappa(T)=\min\{\kappa(T_v):v\in M^*(T)\}$ and, for $k \in [1,\kappa(T)]$,
\begin{equation}{\label{tree_main-v2}}
\dim_f^k(T)=k|M_2(T)|+\displaystyle\sum_{v\in M_3(T)}\dim_f^k(T_v).
\end{equation}
\end{theorem}

\begin{proof}
If $ex(T)=1$, then $|M_2(T)|=0$ and $|M_3(T)|=1$, and so (\ref{tree_main-v2}) trivially holds; see Propostion~\ref{fkdim_tree} for explicit formulas for $\kappa(T)$ and $\dim_f^k(T)$. So, let $ex(T) \ge 2$; then $|M^*(T)| \ge 2$. By Theorem~\ref{kappa_tree}, $\kappa(T)=\displaystyle\min\{\kappa(T_v):v\in M^*(T)\}$. Let $k \in [1, \kappa(T)]$; notice that $\kappa(T) \le \kappa(T_z)$ for any $z\in M^*(T)$. 

First, we show that $\dim _f^k(T) \ge k|M_2(T)|+\sum_{v\in M_3(T)}\dim_f^k(T_v)$. For $v \in M_2(T)$, let $N(v) \cap V(T_{v})=\{r_1, r_2\}$. For $w \in M_3(T)$, let $\ell_1, \ell_2, \ldots, \ell_{\sigma}$ be the terminal vertices of $w$ with $ter_T(w)=\sigma$, and let $s_i$ be the neighbor of $w$ lying on the $w-\ell_i$ path, where $i \in \{1,2,\ldots, \sigma\}$; further, let $P^{i}$ denote the $s_{i}-\ell_{i}$ path. Let $g: V(T) \rightarrow [0,1]$ be a minimum $k$-resolving function of $T$. If $M_2(T) \neq \emptyset$, then, for any $v \in M_2(T)$, $R\{r_1, r_2\}=V(T_{v})-\{v\}$ and $g(R\{r_1, r_2\})=g(V(T_{v})-\{v\}) \ge k$; thus $\sum_{v\in M_2(T)} g(V(T_{v})) \ge  \sum_{v\in M_2(T)}k=k|M_2(T)|$. If $M_3(T) \neq \emptyset$, then, for any $w \in M_3(T)$, notice $R\{s_{i}, s_{j}\}=V(P^{i}) \cup V(P^{j})$ for any distinct $i,j \in \{1,2,\ldots, \sigma\}$. This, together with the argument used in the proof of Proposition~\ref{fkdim_tree}, we have $\sum_{w\in M_3(T)}g(V(T_{w}))\ge \sum_{w\in M_3(T)}\dim_f^k(T_{w})$. Thus, we have $$g(V(T)) \ge \sum_{v \in M_2(T)} g(V(T_v)) +\sum_{w \in M_3(T)}g(V(T_{w})) \ge k|M_2(T)|+\sum_{w \in M_3(T)} \dim_f^k (T_w).$$

Next, we show that $\dim _f^k(T) \le k|M_2(T)|+\sum_{v\in M_3(T)}\dim_f^k(T_v)$. For each $w\in M_3(T)$, let $g_w$ be a minimum $k$-resolving function on $V(T_w)$. For each $w\in M_2(T)$, define a function $h_w$ on $V(T_w)$ such that $h_w(w)=0$ and $h_w(u)=\frac{k}{|V(T_w)|-1}$ if $u\neq w$. For $k \in [1, \kappa(T)]$, let $g: V(T) \rightarrow [0,1]$ be the function defined by 
\begin{equation*}
g(u)=\left\{
\begin{array}{ll}
g_w(u) & \mbox{if } u \in V(T_w) \mbox{ for }w \in M_3(T),\\
h_w(u) & \mbox{if } u \in V(T_w) \mbox{ for }w \in M_2(T),\\
0 & \mbox{otherwise}.
\end{array}\right.
\end{equation*}
Note that (i) for each $w\in M_2(T)$, $g(V(T_w))=h_w(V(T_w)-\{w\})=k$; (ii) for each $w\in M_3(T)$, $g(V(T_w))=g_w(V(T_w))=\dim_f^k(T_w) \ge k$; (iii) $g(V(T))=k|M_2(T)|+\sum_{w\in M_3(T)} \dim_f^k(T)$. It suffices to show $g$ is a $k$-resolving function of $T$. Obviously, $0 \le g(u) \le 1$ for each $u \in V(T)$. So, let $x$ and $y$ be distinct vertices of $T$; we will show that $g(R\{x,y\})\geq k$. Consider three cases: (1) there is a $w\in M^*(T)$ such that $\{x,y\}\subseteq V(T_w)$; (2) there is a $w\in M^*(T)$ such that $x\in V(T_w)$ and $y\notin V(T_w)$; (3) $\{x,y\}\subseteq V(T)-\cup_{w\in M^*(T)}V(T_w)$. In case (1), if $w\in M_3(T)$, then $g(R\{x,y\}) \ge g_w(R\{x,y\} \cap V(T_w))\geq k$, since $g_w$ is a $k$-resolving function on $V(T_w)$; if $w\in M_2(T)$ and $d(x,w)\neq d(y,w)$, then there is a $z\in M^*(T)-\{w\}$ such that $g(R\{x,y\})\geq g(V(T_z))\geq k$; if $w\in M_2(T)$ and $d(x,w)=d(y,w)$, then $g(R\{x,y\})=h_w(V(T_w)-\{w\})=k$. In case (2), by Lemma~\ref{lemma_tree}, either $R\{x,y\} \supseteq V(T_w)$ or $R\{x,y\} \supseteq V(T_{w'})$ for some $w' \in M^*(T)-\{w\}$; thus $g(R\{x,y\}) \ge k$. So, we consider case (3). Note that $x\in V(T_{x'})$ for some $x' \in M_1(T)$ or $x\not\in V(T_z)$ for any $z\in M(T)$; similarly, $y\in V(T_{y'})$ for some $y' \in M_1(T)$ or $y\not\in V(T_{z'})$ for any $z'\in M(T)$. If $\{x,y\} \subseteq V(T_v)$ for some $v\in M_1(T)$, then $d(v,x) \neq d(v,y)$ and there exist distinct $v', v'' \in M^*(T)$ such that $v$ lies on the $v'-v''$ path in $T$; thus $g(R\{x,y\}) \ge g(V(T_{v'}))+g(V(T_{v''})) \ge 2k$. If $\{x,y\} \not\subseteq V(T_v)$ for any $v\in M(T)$, there exist distinct $w_1, w_2 \in M^*(T)$ such that both $x$ (or $x'$) and $y$ (or $y'$) lie on the $w_1-w_2$ path in $T$; then $d(w_1,x)=d(w_1,y)$ and $d(w_2,x)=d(w_2,y)$ imply either $x=y$ or $\{x,y\} \subseteq V(T_{s})$ for some $s\in M_1(T)$, where both possibilities contradict the present assumptions. Thus $R\{x,y\} \supseteq V(T_{w_i})$ for at least one $i\in\{1,2\}$, and $g(R\{x,y\}) \ge g(V(T_{w_i})) \ge k$.~\hfill
\end{proof}

Next, we provide an example showing that $\dim_f^k(G) - k\dim_f(G)$ can be arbitrarily large for some $k \in (1,\kappa(G)]$. 

\begin{remark}
The value of $\dim_f^k(G) - k\dim_f(G)$ can arbitrarily large, as $G$ varies, for some $k \in (1,\kappa(G)]$. Let $T$ be a tree with $ex(T)=1$. Let $v$ be the exterior major vertex of $T$ and let $\ell_1, \ell_2, \ldots, \ell_{\alpha}$ be the terminal vertices of $T$ such that $d(v, \ell_1)=1<\beta=d(v, \ell_j)$ for each $j \in \{2,3,\ldots, \alpha\}$, where $\alpha \ge 3$. By Proposition~\ref{fkdim_tree}, $\kappa(T)=\beta+1$ and $\dim_f^{\beta+1}(T)=(\alpha-1)(\beta+1)-(\alpha-2)=(\alpha-1)\beta+1$. Since $\dim_f(T)=\frac{\alpha}{2}$ by Theorem~\ref{thm_frac}(b), $\dim_f^{\beta+1}(T)-(\beta+1)\dim_f(T)=(\alpha-1)\beta+1-(\beta+1)(\frac{\alpha}{2})=(\frac{\alpha}{2}-1)(\beta-1)$ can be arbitrarily large, as $\alpha$ or $\beta$ gets big enough.
\end{remark}

Next, we determine the fractional $k$-metric dimension of cycles.

\begin{proposition}\label{fkdim_cycle}
Let $C_n$ be an n-cycle, where $n \ge 3$. Then
\begin{equation}\label{eq_cycle}
\dim_f^k(C_n)=k \dim_f(C_n)=\left\{
\begin{array}{ll}
\frac{kn}{n-2} & \mbox{ if $n$ is even and }k\in[1,n-2],\\[\smallskipamount]
\frac{kn}{n-1} & \mbox{ if $n$ is odd and }k\in[1,n-1].
\end{array}\right.
\end{equation}
\end{proposition}

\begin{proof}
Note that $\kappa(C_n)=n-2$ for an even $n$, and $\kappa(C_n)=n-1$ for an odd $n$. Let $k \in [1,\kappa(C_n)]$. For an even $n \ge 4$, a function $g:V(C_n) \rightarrow [0,1]$ defined by $g(u)=\frac{1}{n-2}$, for each $u \in V(C_n)$, is a minimum resolving function of $C_n$: (i) $0 < g(u)=\frac{1}{n-2} \le \frac{1}{k} \le 1$ since $n \ge 4$; (ii) for distinct $x,y \in V(C_n)$, $|R\{x,y\}| \ge n-2$, and thus $g(R\{x,y\}) \ge (n-2) (\frac{1}{n-2})=1$; (iii) $g(V(C_n))=\frac{n}{n-2}=\dim_f(C_n)$ by Theorem~\ref{thm_frac}(d). Similarly, for an odd $n \ge 3$, one can easily check that a function $h:V(C_n) \rightarrow [0,1]$ defined by $h(u)=\frac{1}{n-1}$, for each $u \in V(C_n)$, is a minimum resolving function of $C_n$ satisfying $h(u) \le \frac{1}{k}$. By Lemma~\ref{fdim_fkdim2} and Theorem~\ref{thm_frac}(d), (\ref{eq_cycle}) follows.~\hfill
\end{proof}

\begin{remark}
Note that, for any fixed $k \in [1, \kappa(C_n)]$, $\lim_{n \rightarrow \infty}\dim_f^k(C_n)=k$ by Proposition~\ref{fkdim_cycle} (c.f. Proposition~\ref{characterization}(a)).
\end{remark}

Next, we determine the fractional $k$-metric dimension of wheel graphs.

\begin{proposition}
For the wheel graph $W_n$ of order $n \ge 5$,
\begin{equation}\label{eq_wheel}
\dim_f^k(W_n)=k \dim_f(W_n)=\left\{
\begin{array}{ll}
2k & \mbox{ if $n=5$ and } k \in [1,2],\\
\frac{3k}{2} & \mbox { if $n=6$ and } k \in [1,4],\\ [\smallskipamount]
\frac{k(n-1)}{4} & \mbox{ if $n \ge 7$ and } k \in [1,4].
\end{array}\right.
\end{equation}
\end{proposition}

\begin{proof}
For $n \ge 5$, the wheel graph $W_n=C_{n-1}+K_1$ is obtained from an $(n-1)$-cycle $C_{n-1}$ by joining an edge from each vertex of $C_{n-1}$ to a new vertex, say $v$; let the $C_{n-1}$ be given by $u_1, u_2, \ldots, u_{n-1}, u_1$. Note that $\diam(W_n)=2$ for $n \ge 5$.

\textbf{Case 1: $n=5$.} Note that $\kappa(W_5)=2$ since $R\{u_1, u_3\}=\{u_1, u_3\}$. Let $k \in [1,2]$. Let $g:V(W_5) \rightarrow [0,1]$ be a function defined by $g(v)=0$ and $g(u_i)=\frac{1}{2}$ for each $i \in \{1,2,3,4\}$. Then $g$ is a minimum resolving function of $W_5$: (i) $0 \le g(x) \le \frac{1}{k} \le 1$ for each $x \in V(W_5)$; (ii) for distinct $i,j \in \{1,2,3,4\}$, $g(R\{u_i, u_j\}) \ge g(u_i)+g(u_j)=1$; (iii) for $i \in \{1,2,3,4\}$, $g(R\{v, u_i\}) \ge g(u_i)+g(u_{\ell})=1$, where $u_{\ell} \in V(W_5)-N[u_i]$; (iv) $g(V(W_5))=2=\dim_f(W_5)$ by Theorem~\ref{thm_frac}(e). By Lemma~\ref{fdim_fkdim2} and Theorem~\ref{thm_frac}(e), $\dim_f^k(W_5)=k\dim_f(W_5)=2k$ for $k \in [1,2]$.

\textbf{Case 2: $n \ge 6$.} First, we show that $\kappa(W_n)=4$ in this case. For each $i \in \{1,2,\ldots, n-1\}$, $R\{v, u_i\}=(V(W_n)-N(u_i)) \cup \{v\}$ with $|R\{v, u_i\}|=n-2 \ge 4$. For distinct $i,j \in \{1,2,\ldots,n-1\}$, (i) if $u_iu_j\in E(W_n)$, then $R\{u_i, u_j\}=(N(u_i) \cup N(u_j))-\{v\}$ with $|R\{u_i, u_j\}|=4$; (ii) if $u_iu_j \not\in E(W_n)$ and $|N(u_i) \cap N(u_j)|=1$, then $R\{u_i, u_j\}=(N[u_i] \cup N[u_j])-\{v\}$ with $|R\{u_i, u_j\}|=6$; (iii) if $u_iu_j \not\in E(W_n)$ and $|N(u_i) \cap N(u_j)|=2$, then $R\{u_i, u_j\}=(N[u_i] \cup N[u_j])-(N(u_i) \cap N(u_j))$ with $|R\{u_i, u_j\}|=4$. So, $\kappa(W_n)=4$ for $n \ge 6$.

Second, we determine $\dim_f^k(W_n)$ for $n \ge 6$. Let $k \in [1,4]$. For $n=6$, a function $g: V(W_6) \rightarrow [0,1]$ defined by $g(x)=\frac{1}{4}$, for each $x \in V(W_6)$, is a minimum resolving function of $W_6$: (i) $0 <g(x) \le \frac{1}{k}\le1$ for each $x\in V(W_6)$; (ii) for any distinct $x,y \in V(W_6)$, $g(R\{x,y\}) \ge 4(\frac{1}{4})=1$; (iii) $g(V(W_6))=\frac{3}{2}=\dim_f(W_6)$ by Theorem~\ref{thm_frac}(e). For $n \ge 7$, let $h: V(W_n) \rightarrow [0,1]$ be a function defined by $h(v)=0$ and $h(u_i)=\frac{1}{4}$ for each $i \in \{1,2,\ldots, n-1\}$. Then $h$ is a minimum resolving function of $W_n$: (i) $0 \le h(x) \le \frac{1}{k}\le1$ for each $x \in V(W_n)$; (ii) for each $i \in \{1,2,\ldots, n-1\}$, $|R\{v, u_i\}| \ge n-2 \ge 5$ since $n \ge 7$, and hence $h(R\{v, u_i\}) \ge 4 (\frac{1}{4})=1$; (iii) for distinct $i,j \in \{1,2,\ldots, n-1\}$, $|R\{u_i, u_j\}| \ge 4$ and $v \not\in R\{u_i, u_j\}$, and thus $h(R\{u_i, u_j\}) \ge 4 (\frac{1}{4})=1$; (iv) $h(V(W_n))=\frac{n-1}{4}=\dim_f(W_n)$ by Theorem~\ref{thm_frac}(e). Therefore, by Lemma~\ref{fdim_fkdim2} and Theorem~\ref{thm_frac}(e), (\ref{eq_wheel}) holds for $n\ge 6$ and $k \in [1,4]$.~\hfill
\end{proof}

Next, we determine the fractional $k$-metric dimension of the Petersen graph.

\begin{proposition}
For the Petersen graph $\mathcal{P}$, $\dim_f^k(\mathcal{P})=k\dim_f(\mathcal{P})=\frac{5}{3}k$ for $k \in [1,6]$.
\end{proposition}

\begin{proof}
Note that $\mathcal{P}$ is 3-regular and vertex-transitive. Since $\diam(\mathcal{P})=2$, any two distinct vertices in $\mathcal{P}$ are either adjacent or at distance two apart.

We first show that $\kappa(\mathcal{P})=6$. For any distinct $x,y \in V(\mathcal{P})$, $R\{x,y\}=N[x] \cup N[y]-(N(x) \cap N(y))$ and $|R\{x,y\}|=6$: (i) if $xy \in E(\mathcal{P})$, then $N(x) \cap N(y)=\emptyset$ and $x \in N[y]$ and $y \in N[x]$; (ii) if $xy \not\in E(\mathcal{P})$, then $|N(x) \cap N(y)|=1$. So, $\kappa(\mathcal{P})=6$.

Now, let $k \in [1,6]$. Since $\dim_f^k(\mathcal{P}) \ge k\dim_f(\mathcal{P})$ by Lemma~\ref{fdim_fkdim}, it suffices to show that $\dim_f^k(\mathcal{P}) \le k\dim_f(\mathcal{P})$. Let $g: V(\mathcal{P}) \rightarrow [0,1]$ be a function defined by $g(v)=\frac{k}{6}$ for each $v \in V(\mathcal{P})$. Since $0 \le g(v) \le 1$ for each vertex $v \in V(\mathcal{P})$ and $g(R\{x,y\})=6(\frac{k}{6}) \ge k$ for any two distinct $x,y \in V(\mathcal{P})$, $g$ is a $k$-resolving function of $\mathcal{P}$. So, $\dim_f^k(\mathcal{P}) \le |V(\mathcal{P})| (\frac{k}{6})=\frac{10k}{6}=k(\frac{5}{3})=k \dim_f(\mathcal{P})$ by Theorem~\ref{thm_frac}(c).~\hfill
\end{proof}

Next, we determine the fractional $k$-metric dimension of a bouquet of cycles.

\begin{proposition}
Let $B_m$ be a bouquet of $m$ cycles $C^1, C^2, \ldots, C^m$ with a cut-vertex (i.e., the vertex sum of m cycles at one common vertex), where $m \ge 2$; further, let $C^1$ be the cycle of the minimum length among the $m$ cycles of $B_m$. Then, for $ k \in [1, \kappa(B_m)]$, $\dim_f^k(B_m)=k \dim_f(B_m)=km$, where 

\noindent $\kappa(B_m)=\left\{
\begin{array}{ll}
|V(C^1)|-1 & \mbox{ if $C^1$ is an odd cycle},\\
|V(C^1)|-2 & \mbox{ if $C^1$ is an even cycle}.
\end{array}\right.$
\end{proposition}

\begin{proof}
Let $v$ be the cut-vertex of $B_m$. For each $i \in \{1,2,\ldots, m\}$, let $C^i$ be given by $v, u_{i,1}, u_{i,2}, \ldots, u_{i, r_i}, v$ and let $P^i =C^i-v$; further, let $P^{i,1}$ be the $u_{i,1}-u_{i, \lceil\frac{r_i}{2}\rceil}$ geodesic. Without loss of generality, let $r_1 \le r_2 \le \ldots \le r_m$.

\vspace{.2cc}

\textbf{Claim 1:} If $C^1$ is an odd cycle, then $\kappa(B_m)=r_1=|V(C^1)|-1$; if $C^1$ is an even cycle, $\kappa(B_m)=r_1-1=|V(C^1)|-2$.

\noindent Proof of Claim 1. Let $x$ and $y$ be distinct vertices of $B_m$. First, let $x,y \in V(C^i)$ for some $i \in \{1,2,\ldots, m\}$. If $d(v,x) \neq d(v,y)$, then $R\{x,y\} \supseteq V(C^j)$ with $|R\{x,y\}| \ge |V(C^j)|=r_j+1\ge r_1+1$ for $j \neq i$. If $d(v,x)=d(v,y)$ and $C^i$ is an odd cycle, then $R\{x,y\}=V(P^i)$ with $|R\{x,y\}|=r_i \ge r_1$; notice, for an odd cycle $C^1$, $|R\{u_{1,1}, u_{1, r_1}\}|=r_1$. If $d(v,x)=d(v,y)$ and $C^i$ is an even cycle, then $R\{x,y\}=V(P^i)-\{u_{i, \lceil\frac{r_i}{2}\rceil}\}$ with $|R\{x,y\}|=r_i -1$, where $r_i-1 \ge r_1$ if $C^1$ is an odd cycle, and $r_i-1 \ge r_1-1$ if $C^1$ is an even cycle; notice, for an even cycle $C^1$, $|R\{u_{1,1}, u_{1, r_1}\}|=r_1-1$.

Second, let $x \in V(P^i)$ and $y \in V(P^j)$ for distinct $i,j \in \{1,2,\ldots,m\}$; let $x \in V(P^{i,1})$ and $y \in V(P^{j,1})$, without loss of generality. If $d(v, x)=d(v, y)$, then $R\{x,y\} \supseteq V(P^{i,1}) \cup V(P^{j,1})$ with $|R\{x,y\}| \ge \lceil\frac{r_i}{2}\rceil+\lceil\frac{r_j}{2}\rceil \ge r_1$. So, let $d(v, x) \neq d(v,y)$, say $d(v,x)<d(v,y)$ without loss of generality; then $d(u,x) \neq d(u,y)$ for each $u \in V(P^{i,1})$, and $r_j \ge 3$. If $r_j=3$, then $R\{x,y\}\supseteq V(P^{i,1})\cup V(P^j)$. If $r_j \ge 4$, then at most two vertices in $C^j$ are at equal distance from $x$ and $y$; thus, $R\{x,y\} \supseteq V(P^{i,1}) \cup (V(P^j)-\{w_1,w_2\})$ such that $w_t \in V(P^j)$ with $d(w_t, x)=d(w_t, y)$, where $t \in \{1,2\}$. In each case, $|R\{x,y\}| \ge \lceil\frac{r_i}{2}\rceil+\lceil\frac{r_j}{2}\rceil \ge r_1$.~$\Box$

\vspace{.2cc}

\textbf{Claim 2:} For $k \in [1, \kappa(B_m)]$, $\dim_f^k(B_m)=k\dim_f(B_m)=km$. 

\noindent Proof of Claim 2. Let $k \in [1, \kappa(B_m)]$, and let $h:V(B_m) \rightarrow [0,1]$ be a function defined by
\begin{equation*}
h(u)=\left\{
\begin{array}{ll}
\frac{1}{r_i} & \mbox{ for each } u \in V(P^i) \mbox{ if $C^i$ is an odd cycle},\\
\frac{1}{r_j-1} & \mbox{ for each } u \in V(P^j)-\{u_{j,\lceil\frac{r_j}{2}\rceil}\} \mbox{ if $C^j$ is an even cycle},\\
0 & \mbox{ otherwise}.
\end{array}\right.
\end{equation*}
Note that, for each $i \in \{1,2,\ldots, m\}$, $h(V(P^i))=1$ if $C^i$ is an odd cycle, and $h(V(P^i))=h(V(P^i))-h({u_{i, \lceil\frac{r_i}{2}\rceil}})=1$ if $C^i$ is an even cycle. We also note that $h$ is a minimum resolving function of $B_m$: (i) $0 \le h(u) \le \frac{1}{k} \le1$ for each $u \in V(B_m)$; (ii) if $x,y \in V(C^i)$ with $d(v, x) \neq d(v,y)$, for some $i \in \{1,2,\ldots, m\}$, then $h(R\{x,y\}) \ge h(V(C^j)) \ge 1$ for $j \neq i$; (iii) if $x,y \in V(C^i)$ with $d(v, x)=d(v,y)$ and $x\neq y$, for some $i \in \{1,2,\ldots, m\}$, then $h(R\{x,y\}) = h(V(P^i))=1$ when $C^i$ is an odd cycle, and $h(R\{x,y\}) = h(V(P^i))-h(u_{i, \lceil\frac{r_i}{2}\rceil})=1$ when $C^i$ is an even cycle; (iv) if $x \in V(P^i)$ and $y \in V(P^j)$ for distinct $i,j \in \{1,2,\ldots,m\}$, then $h(R\{x,y\}) \ge \frac{1}{2}h(V(P^i))+\frac{1}{2}h(V(P^j)) \ge 1$ using a similar argument used in the proof of Claim 1; (v) $h(V(B_m))\!=\!m\!=\!\dim_f(B_m)$ by  Theorem~\ref{thm_frac}(f). So, by Lemma~\ref{fdim_fkdim2} and Theorem~\ref{thm_frac}(f), $\dim_f^k(B_m)=k\dim_f(B_m)=km$ for $k \in [1, \kappa(B_m)]$.~$\Box$~\hfill
\end{proof}

Next, we determine the fractional $k$-metric dimension of complete multi-partite graphs.

\begin{proposition}\label{fkdim_multipartite}
For $m \ge 2$, let $G=K_{a_1,a_2, \ldots, a_m}$ be a complete $m$-partite graph of order $n=\sum_{i=1}^{m}a_i \ge 3$, and let $s$ be the number of partite sets of $G$ consisting of exactly one element. Then, for $k \in [1,2]$,
\begin{equation*}
\dim^k_f(G)=k\dim_f(G)=\left\{
\begin{array}{ll}
\frac{k(n-1)}{2} & \mbox{if $s=1$},\\
\frac{kn}{2} & \mbox{otherwise}.
\end{array}\right.
\end{equation*}
\end{proposition}

\begin{proof}
Let $V(G)$ be partitioned into $m$-partite sets $V_1, V_2, \ldots, V_m$ with $|V_i|=a_i$, where $i \in \{1,2,\ldots, m\}$. Without loss of generality, let $a_1 \le a_2 \le \ldots \le a_m$. Note that $\kappa(G)=2$: (i) if $a_m \ge 2$, then, for two distinct $x,y \in V_m$, $R\{x,y\}=\{x,y\}$; (ii) if $a_m=1$, then, for $x\in V_1$ and $y \in V_2$, $R\{x,y\}=\{x,y\}$. Let $k \in [1,2]$.

First, let $s \neq 1$. A function $g: V(G) \rightarrow [0,1]$ defined by $g(v)=\frac{1}{2}$, for each $v \in V(G)$, is a minimum resolving function of $G$: (i) $0<g(v) \le \frac{1}{k} \le 1$ for each $v\in V(G)$; (ii) for any distinct vertices $x, y \in V(G)$, $g(R\{x,y\}) \ge g(x)+g(y)=1$; (iii) $g(V(G))=\frac{n}{2}=\dim_f(G)$ by Theorem~\ref{thm_frac}(g). So, $\dim_f^k(G)=k\dim_f(G)=\frac{kn}{2}$ by Lemma~\ref{fdim_fkdim2} and Theorem~\ref{thm_frac}(g).

Second, let $s=1$. Let $h: V(G) \rightarrow [0,1]$ be a function defined by $h(u)=0$ for $u \in V_1$ and $h(v)=\frac{1}{2}$ for each $v \in V(G)-V_1$. Then $h$ is a minimum resolving function of $G$ : (i) $0 \le h(v) \le \frac{1}{k} \le 1$ for each $v\in V(G)$; (ii) for any two distinct vertices $x, y \in V(G)-V_1$, $h(R\{x,y\}) \ge h(x)+h(y) =1$; (iii) for $x \in V_1$ and $y \in V_i \subseteq V(G)-V_1$, $h(R\{x,y\}) \ge h(V_i) \ge 1$, where $i \in \{2, \ldots, m\}$; (iv) $h(V(G))=\frac{n-1}{2}=\dim_f(G)$ by Theorem~\ref{thm_frac}(g). So, $\dim_f^k(G)=k\dim_f(G)=\frac{k(n-1)}{2}$ for $k \in [1,2]$ by Lemma~\ref{fdim_fkdim2} and Theorem~\ref{thm_frac}(g).~\hfill
\end{proof}

Now, we consider the fractional $k$-metric dimension of grid graphs (i.e., the Cartesian product of two paths). The \emph{Cartesian product} of two graphs $G$ and $H$, denoted by $G \square H$, is the graph with the vertex set $V(G) \times V(H)$ such that $(u,w)$ is adjacent to $(u',w')$ if and only if either $u=u'$ and $ww' \in E(H)$, or $w=w'$ and $uu' \in E(G)$. See Figure~\ref{fig_grid} for the labeling of $P_6 \square P_4$. 

\begin{figure}[ht]
\centering
\begin{tikzpicture}[scale=.7, transform shape]

\node [draw, shape=circle] (4) at  (0,4.5) {};
\node [draw, shape=circle] (3) at  (0,3) {};
\node [draw, shape=circle] (2) at  (0,1.5) {};
\node [draw, shape=circle] (1) at  (0,0) {};
\node [draw, shape=circle] (44) at  (1.5,4.5) {};
\node [draw, shape=circle] (33) at  (1.5,3) {};
\node [draw, shape=circle] (22) at  (1.5,1.5) {};
\node [draw, shape=circle] (11) at  (1.5,0) {};
\node [draw, shape=circle] (444) at  (3,4.5) {};
\node [draw, shape=circle] (333) at  (3,3) {};
\node [draw, shape=circle] (222) at  (3,1.5) {};
\node [draw, shape=circle] (111) at  (3,0) {};
\node [draw, shape=circle] (4444) at  (4.5,4.5) {};
\node [draw, shape=circle] (3333) at  (4.5,3) {};
\node [draw, shape=circle] (2222) at  (4.5,1.5) {};
\node [draw, shape=circle] (1111) at  (4.5,0) {};
\node [draw, shape=circle] (44444) at  (6,4.5) {};
\node [draw, shape=circle] (33333) at  (6,3) {};
\node [draw, shape=circle] (22222) at  (6,1.5) {};
\node [draw, shape=circle] (11111) at  (6,0) {};
\node [draw, shape=circle] (444444) at  (7.5,4.5) {};
\node [draw, shape=circle] (333333) at  (7.5,3) {};
\node [draw, shape=circle] (222222) at  (7.5,1.5) {};
\node [draw, shape=circle] (111111) at  (7.5,0) {};

\draw(1)--(2)--(3)--(4);
\draw(11)--(22)--(33)--(44);
\draw(111)--(222)--(333)--(444);
\draw(1111)--(2222)--(3333)--(4444);
\draw(11111)--(22222)--(33333)--(44444);
\draw(111111)--(222222)--(333333)--(444444);
\draw(1)--(11)--(111)--(1111)--(11111)--(111111);
\draw(2)--(22)--(222)--(2222)--(22222)--(222222);
\draw(3)--(33)--(333)--(3333)--(33333)--(333333);
\draw(4)--(44)--(444)--(4444)--(44444)--(444444);

\node [scale=1] at (-0.7,0) {$(1,1)$};
\node [scale=1] at (-0.7,1.5) {$(1,2)$};
\node [scale=1] at (-0.7,3) {$(1,3)$};
\node [scale=1] at (-0.7,4.5) {$(1,4)$};

\node [scale=1] at (1.5,-0.4) {$(2,1)$};
\node [scale=1] at (2,1.8) {$(2,2)$};
\node [scale=1] at (2,3.3) {$(2,3)$};
\node [scale=1] at (1.5,4.9) {$(2,4)$};

\node [scale=1] at (3,-0.4) {$(3,1)$};
\node [scale=1] at (3.5,1.8) {$(3,2)$};
\node [scale=1] at (3.5,3.3) {$(3,3)$};
\node [scale=1] at (3,4.9) {$(3,4)$};

\node [scale=1] at (4.5,-0.4) {$(4,1)$};
\node [scale=1] at (5,1.8) {$(4,2)$};
\node [scale=1] at (5,3.3) {$(4,3)$};
\node [scale=1] at (4.5,4.9) {$(4,4)$};

\node [scale=1] at (6,-0.4) {$(5,1)$};
\node [scale=1] at (6.5,1.8) {$(5,2)$};
\node [scale=1] at (6.5,3.3) {$(5,3)$};
\node [scale=1] at (6,4.9) {$(5,4)$};

\node [scale=1] at (8.15,0) {$(6,1)$};
\node [scale=1] at (8.15,1.5) {$(6,2)$};
\node [scale=1] at (8.15,3) {$(6,3)$};
\node [scale=1] at (8.15,4.5) {$(6,4)$};

\end{tikzpicture}
\caption{Labeling of $P_6 \square P_4$.}\label{fig_grid}
\end{figure}
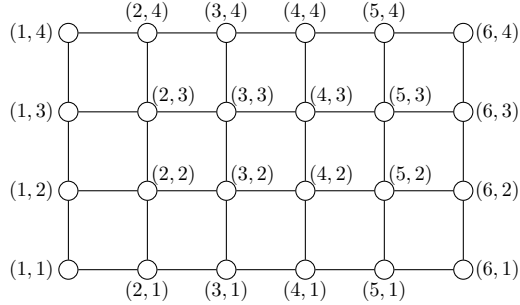
We recall the following result.

\begin{theorem}\emph{~\cite{kdim2}}\label{grid_kdim}
For $s,t \ge 2$, $\kappa(P_s \square P_t)=s+t-2$ and $\dim^k(P_s \square P_t)=2k$, where $k \in \{1,2, \ldots, s+t-2\}$.
\end{theorem}

\begin{proposition}\label{grid_fkdim}
For $k \in [1, s+t-2]$, $\dim_f^k(P_s \square P_t)=k\dim_f(P_s \square P_t)=2k$, where $s,t \ge 2$.
\end{proposition}

\begin{proof}
Let $s \ge t \ge 2$, and let $G=P_s \square P_t$ and $L=\{v\in V(G): 2 \le \deg(v) \le 3\}$. By Theorem~\ref{grid_kdim}, $\kappa(G)=s+t-2$. Let $k \in [1,s+t-2]$. Since $\dim^k_f(G) \ge k \dim_f(G)=2k$ by Lemma~\ref{fdim_fkdim} and Theorem~\ref{thm_frac}(h), it suffices to show that $\dim_f^k(G) \le 2k$. Let $g: V(G) \rightarrow [0,1]$ be a function defined by
\begin{equation*}
g(v)=\left\{
\begin{array}{ll}
\frac{k}{s+t-2} & \mbox{ if } v\in L,\\
0 & \mbox{ otherwise. }
\end{array}\right.
\end{equation*}
Note that $g(V(G))=2k$. We will show that $g$ is a $k$-resolving function of $G$. Clearly, $0 \le g(v) \le 1$ for each $v \in V(G)$. Let $x=(a, b)$ and $y=(c, d)$ be two distinct vertices of $G$. We consider two cases.

\textbf{Case 1: $a=c$ or $b=d$.} If $a=c$, then $R\{x,y\} \cap L\supseteq \cup_{i=1}^{s}\{(i, 1), (i, t)\}$, and thus $g(R\{x,y\}) \ge (2s)(\frac{k}{s+t-2}) \ge k$ since $s \ge t \ge 2$. So, let $b=d$; without loss of generality, let $a<c$. Let $z=(\alpha, \beta) \in L$. Note that (i) if $\alpha \le a$, then $d(z,x)<d(z,y)$ and thus $R\{x,y\} \cap L \supseteq \cup_{j=1}^{t}\{(1, j)\} \cup (\cup_{i=2}^{a}\{(i,1),(i,t)\})$; (ii) if $\alpha \ge c$, then $d(z,x)>d(z,y)$ and thus $R\{x,y\} \cap L \supseteq \cup_{j=1}^{t}\{(s, j)\}\cup (\cup_{i=c}^{s-1}\{(i,1),(i,t)\})$; (iii) if $a<\alpha<c$ and $\beta=1$, then there exists at most one such $z\in L$ satisfying $d(z,x)=d(z,y)$, since $d(z,x)=d(z,y)$ implies $\alpha-a+b-1=c-\alpha+b-1$, i.e., $2\alpha=a+c$; similarly, if $a<\alpha<c$ and $\beta=t$, then there exists at most one such $z \in L$ satisfying $d(z,x)=d(z,y)$. Thus, $|R\{x,y\} \cap L| \ge 2s+2t-6$, and hence $g(R\{x,y\}) \ge (2s+2t-6)(\frac{k}{s+t-2})=2k(\frac{s+t-3}{s+t-2}) \ge k$, since $s+t \ge 4$.

\textbf{Case 2: $a \neq c$ and $b \neq d$.} Without loss of generality, let $a<c$; further, assume that $b<d$ (the case for $b>d$ can be handled similarly). Let $z'=(\alpha', \beta') \in L$. Note that (i) if $\alpha' \le a$ and $\beta'=1$, then $d(z',x)<d(z',y)$ and thus $R\{x,y\} \cap L \supseteq \cup_{i=1}^{a}\{(i, 1)\}$; (ii) if $\alpha' \ge c$ and $\beta'=t$, then $d(z',x)>d(z',y)$ and thus $R\{x,y\} \cap L \supseteq \cup_{i=c}^{s}\{(i, t)\}$; (iii) if $a<\alpha'<c$ (i.e., $c \neq a+1$) and $\beta'=1$, then there exists at most one such $z' \in L$ satisfying $d(z',x)=d(z',y)$, since $d(z',x)=d(z',y)$ implies $\alpha'-a+b-1=c-\alpha'+d-1$, i.e., $2\alpha'=a-b+c+d$; similarly, if $a<\alpha'<c$ and $\beta'=t$, there exists at most one such $z' \in L$ satisfying $d(z',x)=d(z',y)$. Likewise, we note that (i) if $\alpha'=1$ and $\beta' \le b$, then $d(z',x)<d(z',y)$ and thus $R\{x,y\} \cap L \supseteq \cup_{j=1}^{b}\{(1,j)\}$; (ii) if $\alpha' =s$ and $\beta' \ge d$, then $d(z',x)>d(z',y)$ and thus $R\{x,y\} \cap L \supseteq \cup_{j=d}^{t}\{(s, j)\}$; (iii) if $\alpha'=1$ and $b<\beta'<d$ (i.e., $d \neq b+1$), then there exists at most one such $z' \in L$ satisfying $d(z',x)=d(z',y)$, since $d(z',x)=d(z',y)$ implies $a-1+\beta'-b=c-1+d-\beta'$, i.e., $2\beta'=-a+b+c+d$; similarly, if $\alpha'=s$ and $b<\beta'<d$, then there exists at most one such $z' \in L$ satisfying $d(z',x)=d(z',y)$. So, if $c=a+1$ or $d=b+1$, then $|R\{x,y\} \cap L| \ge s+t-2$; if $c \ge a+2$ and $d \ge b+2$, then $|R\{x,y\} \cap L| \ge a+(s-c+1)+2(c-a-2)+b+(t-d+1)+2(d-b-2) \ge s+t-2$. In each case, $g(R\{x,y\}) \ge (s+t-2)(\frac{k}{s+t-2}) = k$.

Thus, in each case, $g$ is a $k$-resolving function of $G$, and hence $\dim_f^k(G) \le g(V(G))=2k$. Therefore, $\dim_f^k(G)=k\dim_f(G)=2k$ for $k \in [1, s+t-2]$ for $s \ge t \ge 2$.~\hfill
\end{proof}


\section{Open Problems}

We conclude this paper with two open problems.\\

\textbf{Problem 1.} Let $\phi(k)=\dim_f^k(G)$ be a function of $k$, for a fixed $G$, on domain $[1, \kappa(G)]$. Is $\phi$ a continuous function of $k$ on every connected graph $G$?\\

\textbf{Problem 2.} Suppose $\dim_f^k(G)$ is given by $\psi(k)$ for integral values of $k$. When and how can we interpolate $\psi$ and deduce $\dim_f^k(G)$ for any real number $k \in [1, \kappa(G)]$?

For example, let $G=P_s \square P_t$, where $s,t \ge 2$. Then $\dim_f^k(G)=2k$ for integers $k \in \{1,2,\ldots, \kappa(G)\}$ by Theorems~\ref{thm_frac}(h), Lemma~\ref{fdim_fkdim}, Observation~\ref{obs2}(b), and Theorem~\ref{grid_kdim}. In Proposition~\ref{grid_fkdim}, we proved that $\dim_f^k(G)=2k$ for any real number $k \in [1, \kappa(G)]$, by using Lemma~\ref{fdim_fkdim} and constructing a $k$-resolving function $g$ on $V(G)$ with $g(V(G))=2k$ for $k \in [1,\kappa(G)]$. The construction of $k$-resolving function for any real number $k \in [1, \kappa(G)]$ in determining $\dim_f^k(G)$ in Proposition~\ref{grid_fkdim} does not appear to carry to the construction of $k$-resolving set for any integral values $k \in \{1,2,\ldots,\kappa(G)\}$ in determining $\dim^k(G)$, and vice versa.

\bigskip  


\textbf{Acknowledgement.} The authors greatly appreciate Dr. Douglas J. Klein for graciously hosting Dr. Ismael G. Yero during his visit to Texas A\&M University at Galveston -- this visit sparked an ongoing collaboration of which the present paper is a product. The authors also thank the anonymous referees for some helpful comments.

\end{document}